\documentclass[reqno]{amsart}
\usepackage{amssymb}
\usepackage{graphicx}
\usepackage[left]{lineno}

\newtheorem{theorem}{Theorem}[section]
\newtheorem*{theorem*}{Theorem}

\newtheorem{proposition}[theorem]{Proposition}

\theoremstyle{definition}

\theoremstyle{remark}
\newtheorem{remark}[theorem]{Remark}
\numberwithin{equation}{section}

\usepackage[margin=2.1cm]{geometry}
\DeclareMathOperator*{\argmin}{arg\,min}
\DeclareMathOperator*{\dom}{dom}

\DeclareMathOperator*{\id}{id}

\author[Alberto Dom\'inguez Corella]{Alberto Dom\'inguez Corella}
\address[Alberto Dom\'inguez Corella]{Institut für Mathematik und Wissenschaftliches Rechnen,
Universität Graz, Heinrichstraße 36, 8010 Graz, Austria}
\address{Institut für Stochastik und Wirtschaftsmathematik,
Technische Universität Wien,
Wiedner Hauptstraße 8, VADOR E105-04, 1040 Wien, Austria}
\email{alberto.of.sonora@gmail.com}

\author[Alejandro Villegas-Acu\~na]{Alejandro Villegas-Acu\~na}
\address[Alejandro Villegas-Acu\~na]{Departamento de Matem\'aticas,
Universidad de Sonora,
Hermosillo, Sonora 83000, M\'exico}
\email{francisco.villegas@unison.mx}

\begin{document}

\title{The  Br\o ndsted-Rockafellar theorem in geodesic spaces}

\address{}
\curraddr{}
\email{}
\thanks{The first author was supported by the Austrian Science Foundation (FWF) under grants P 36344-N and F 100800, and the second author by the SECIHTI under Grant No.~921835.}


\date{}

\dedicatory{}

\commby{}

\begin{abstract}
	We present a constructive version of the Br\o ndsted–Rockafellar theorem in general geodesic metric spaces. Applications include a constructive form of the Caristi theorem and quantitative relations between metric slope error bounds and the global growth of functionals.
\end{abstract}



\maketitle
\setcounter{tocdepth}{1}
\tableofcontents

\section{Introduction, main results and state of the art}

\subsection{Motivation, the classical Br\o ndsted-Rockafellar theorem and summary of results} 
Any measurement comes with finite precision, and the \textit{true state} of a system can only be observed up to a tolerance.  For this reason, approximate solutions are often as good as, and frequently indistinguishable, from true ones. Since mathematical laws governing physical systems are typically formulated in variational form, it is natural to consider \textit{almost critical points}.
\smallbreak\noindent
The Brøndsted–Rockafellar theorem occupies a central place in variational analysis. It establishes that every approximate minimizer of a convex functional can be \textit{lifted} to a nearby point admitting a subgradient of small norm—that is, an almost critical point.
\begin{theorem*}[Br\o ndsted--Rockafellar {\cite[Lemma on page 608]{BR_1965}}]\label{thmBR}
Let \(X\) be a real Banach space and
\(\Phi:X\to \mathbb R\cup\{+\infty\}\) a proper convex lower
semicontinuous functional bounded from below. 
Let \(u\in X\) and \(\varepsilon>0\) be such that 
\[
\Phi(u)\le \inf_{v\in X}\Phi(v)+\varepsilon.
\]
Then, for any \(\lambda>0\), there exist \(\hat u\in X\) and \(\xi\in\partial \Phi(\hat u)\) satisfying 
\begin{align*}
(a)\,\,\, \|u-\hat u\|\le \lambda \qquad 
(b)\,\,\, \Phi(\hat u)\le \Phi(u) \qquad 
(c)\,\,\, \|\xi\|\le \frac{\varepsilon}{\lambda}.
\end{align*}
\end{theorem*}
\noindent
One of the main drawbacks of this result is that it is not constructive; its original proof relied on an order argument via Zorn’s lemma, adapted from the proof of the celebrated Bishop–Phelps lemma \cite[Lemma 1 and Theorem 2]{BP_1963}. Modern proofs use Ekeland’s variational principle \cite[Theorem 1.1]{E_1974}, which does not require the axiom of choice. 
\smallbreak\noindent
Recently, in \cite[Theorem~4.1]{C_2023}, a constructive version of the theorem was established in Hilbert spaces. The result provides the desired point by means of the resolvent of the subdifferential. The proof relies on what is now referred to as \emph{Carlier's inequality} \cite{AD_2025,BSW_2023,BM_2025}, an estimate relating the Fenchel--Young gap to the squared resolvent residual.  The construction in  {\cite[Theorem~4.1]{C_2023}} yields that the choices
\begin{align}\label{constructions}
        \hat u = \Big(\mathrm{id}+\frac{\lambda^2}{\varepsilon}\,\partial\Phi\Big)^{-1}(u)
    \qquad\text{and}\qquad 
    \xi = \frac{\varepsilon}{\lambda^2}\,(u-\hat u)
\end{align}
satisfy items $(a)$-$(c)$ in the theorem stated above. An advantage of this explicit construction is that it enables more refined arguments in applications of the theorem. It also fits naturally within the frameworks of constructive mathematics and computable analysis.
\smallbreak\noindent
The Brøndsted--Rockafellar theorem has been extended in many directions, including operator-theoretic \cite{B_2013,IS_2012,MS_2008,S_2008}
 and convex-analytic \cite{AO_2024,CR_2003,CHP_2018,FI_2024,L_2015,Z_2023}.
But so far only \cite[Theorem~4.1]{C_2023} has given an explicit statement where a construction is provided; a similar construction was used in the proof of \cite[Proposition~2.3]{L_2015}.
\smallbreak\noindent
This paper is devoted to a constructive extension of the Br\o ndsted--Rockafellar theorem to geodesic metric spaces. Such spaces include, for instance, normed linear spaces endowed with the distance induced by the norm, the hyperbolic space equipped with its Riemannian metric (and, more generally, $\operatorname{CAT}(0)$ spaces), and the space of probability measures with finite second moment on the Euclidean space endowed with the Wasserstein distance. The extension to this latter space is particularly significant, since a notion of subdifferential is available; see \cite[Section~10]{AGS_2008} and \cite[Section~4]{AS_2007}.
\smallbreak\noindent
The contributions of this paper can be summarized as follows. 
\begin{itemize}
    \item[(1)] \textbf{Extension to geodesic spaces}. We provide two versions of the Br\o ndsted--Rockafellar theorem. The first one comes in the form of Ekeland’s variational principle (Theorem~\ref{thm1} below), while the second one (Theorem~\ref{thm2} below) is more faithful to the original theorem, replacing the bound on the subdifferential by a bound on the metric slope. The resolvent construction in~(\ref{constructions}) is replaced by proximal points (in Hilbert spaces there is no distinction between resolvent and proximal points). We circumvent Carlier's inequality, which is intrinsically related to the linear structure due to the use of the Legendre-Fenchel transform, through purely metric arguments. We give some pertinent examples illustrating the applicability of the results.
\smallbreak 
  \item[(2)] \textbf{Caristi fixed point theorem.}  As an application of Theorem~\ref{thm1}, we recover a metric version of the Caristi fixed point theorem for geodesically convex functionals. The result (stated in Proposition~\ref{P1} below) is constructive and relies on the existence of proximal points.
\smallbreak 
  \item[(3)] \textbf{A result on global error bounds.}  We show (in Proposition~\ref{P2} below) that if a geodesically convex functional satisfies an error bound controlling the distance to the minimizer set by a power of the metric slope, then it necessarily satisfies a quantitative growth estimate. 
This extends to geodesic metric spaces the classical result, known for convex functionals in Banach spaces,  that the metric sub-regularity of the subdifferential implies a polynomial growth condition at the set of minimizers. 
\smallbreak
  \item[(4)] \textbf{A construction in the Wasserstein space.} 
In Section~\ref{Wassres}, we provide more refined results than the general theorems in this section for the particular case of the space of probability measures with the Wasserstein distance. 
In Theorem~\ref{thmW}, we obtain an analogue of the construction in~(\ref{constructions}), formulated in terms of optimal transport maps and the minimizing movement scheme. 
Furthermore, in Theorem~\ref{Ent_Loj}, we establish the equivalence between the metric sub-regularity of the subdifferential and the so-called entropy–transportation inequality (also referred to as the Talagrand inequality).
\end{itemize}
Below in this section we present the main results; their proofs are given in the next section, and the results specific to the Wasserstein space are collected in Section~\ref{Wassres}.


\subsection{The theorem in geodesic spaces and some remarks}\label{Mr}
Throughout this subsection, $\big(X, d\big)$ denotes a geodesic space; that is, a metric space such that for every $u,v\in X$ there exists a constant-speed geodesic $\gamma:[0,1]\to X$ satisfying $\gamma(0)=u$ and $\gamma(1) = v$. We also consider a proper geodesically convex functional \(\Phi \colon X \to \mathbb{R} \cup \{+\infty\}\) bounded from below. The convexity here is understood as 
\begin{align*}
   u,v\in X\quad \implies \quad\Phi(\gamma(t)) \leq (1-t)\,\Phi(\gamma(0)) + t\,\Phi(\gamma(1))\quad \forall t\in[0,1]
\end{align*}
for at least one constant-speed geodesic $\gamma:[0,1]\to X$ satisfying $\gamma(0)=u$ and $\gamma(1)=v$. We fix numbers $p,q\in (1,+\infty)$ such that $p^{-1} + q^{-1} = 1$.

\subsubsection{The theorem in variational form}
Consider the set-valued mapping $\operatorname{prox_{\Phi}}:X\rightrightarrows X$ given by 
     \[
     \operatorname{prox_{\Phi}}(u):= \argmin_{v\in X}\big\{\Phi(v) + \frac{1}{p}\hspace{0.02cm} d(v,u)^p\big\}.
     \]
This operator is usually referred to as the \textit{proximal mapping}, and in some contexts it can be interpreted as a resolvent operator. The main result of this subsection establishes that the proximal mapping naturally yields points satisfying the conditions of Ekeland’s variational principle.
\begin{theorem}\label{thm1}
Let \(u \in X\) and $\varepsilon>0$ be such that 
\[
\Phi(u) \leq \inf_{v \in X} \Phi(v) + \varepsilon.
\]
Then, for any $\lambda>0$ and  $\hat u\in  \operatorname{prox}_{\tfrac{\lambda^p}{\varepsilon}\Phi}(u)$, the  following estimates are satisfied. 
\begin{align*}
    (a)\,\,\, d(u,\hat u)\le\lambda\qquad\,\,\,\, (b)\,\,\,\Phi(\hat{u})  \leq \Phi(u)\qquad \,\,\,\,(c)\,\,\, \Phi(\hat u) \le \Phi(v) + \frac{\varepsilon}{\lambda} d(v,\hat u)\quad \forall v\in X.
\end{align*}
\end{theorem}
\smallbreak\noindent
We now make some remarks.
\begin{remark}
     While Theorem~\ref{thm1} does not require completeness of the underlying space, it does rely on the existence of minimizers of a regularized problem.  An advantage of it is that it gives an explicit and computable representation of the approximate minimizer, making the result more constructive rather than purely existential. The assumption that proximal points exist is common in the literature, as it simplifies the analysis and avoids technical complications concerning the well-posedness of minimizing movement schemes; e.g., \cite[Assumption~(10.1.1b)]{AGS_2008}.
 \end{remark}

\begin{remark}
We observe that Theorem~\ref{thm1} does not require the functional to be lower semicontinuous at every point. However, whenever a proximal point exists, the functional is automatically lower semicontinuous at that point.
\end{remark}

\begin{remark}
    The point $\hat u$ in Theorem~\ref{thm1} can be viewed as a step of the so-called $p$-generalized minimizing movement scheme in the calculus of variations framework; see, e.g., \cite[Definition 2.0.6]{AGS_2008}. 
\end{remark}

\begin{remark}
	In normed spaces, the inequality in Theorem~\ref{thm1}-$(c)$ is equivalent to the existence of an element $\xi\in X^*$ satisfying $\xi\in\partial\Phi(\hat u)$ and $\|\xi\|\le\varepsilon/\lambda$. A refinement of Theorem~\ref{thm1} in the particular case of normed spaces is given in Appendix~\ref{A}.
\end{remark}

\begin{remark}\label{rem:endpoints}
Theorem~\ref{thm1} is stated for exponents strictly between one and infinity. The remaining cases are discussed in the appendix. 
The limiting regime in which the proximal regularization becomes a hard constraint is treated in Appendix~\ref{B}. 
Appendix~\ref{C} presents a more general theorem based on arbitrary convex proximal kernels, which includes, in particular, the case of exponent one.
\end{remark}

\subsubsection{The theorem with the metric slope}

 The \textit{metric slope} of $\Phi$ at a point $u\in X$ is given by 
\begin{align}\label{mesl}
     |\partial \Phi|(u) := \max\left\{\limsup_{v\longrightarrow u}\frac{\Phi(u)-\Phi(v)}{d(u,v)},0\right\}.
\end{align}
Quantity \eqref{mesl} is always nonnegative, and vanishes exactly at the  minimizers of $\Phi$. If $X$ is a real normed space, then the metric slope coincides with the distance of the convex subdifferential to zero, i.e., $|\partial \Phi|(u) =\operatorname{dist}\big(0,\partial \Phi(u)\big)$ for all $u\in X$, where $\partial\Phi:X\rightrightarrows X^*$ denotes the convex differential of $\Phi$; see \cite[Proposition 2.1 $(vii)$]{K_2015}. This identity is one of the reasons for the term \textit{metric slope}. 
\smallbreak 
We now give an alternative version of Theorem \ref{thm1}.

\begin{theorem}\label{thm2}
	 Let \(u \in X\) and $\varepsilon>0$ be such that 
	\[
	\Phi(u) \leq \inf_{v \in X} \Phi(v) + \varepsilon.
	\]
Then, for any $\lambda>0$ and  $\hat u\in  \operatorname{prox}_{\tfrac{\lambda^p}{\varepsilon}\Phi}(u)$, the  following estimates are satisfied.  
	\begin{align*}
		(a)\,\,\, d(u,\hat u)\le\lambda\qquad\,\,\,\, (b)\,\,\,\Phi(\hat{u}) + \frac{\lambda^q}{p\varepsilon^{\frac{q}{p}}} |\partial\Phi|(\hat{u})^q \leq \Phi(u)\qquad \,\,\,\,(c)\,\,\,|\partial\Phi|(\hat{u}) \leq \frac{\varepsilon}{\lambda}.
	\end{align*}
\end{theorem}
\noindent
We now make a couple of remarks.
\begin{remark}\label{R2}
Interpreting $\Phi$ as a potential energy and  $(u,|\partial\Phi|(u))$ as a position--speed pair, the inequality in Theorem~\ref{thm2}-$(b)$ takes the form of an energy balance. The functional 
\begin{align*}
    X\ni v \longmapsto  \Phi(v) + \frac{\lambda^q}{p\varepsilon^{\frac{q}{p}}} \,|\partial\Phi|(v)^q \in\mathbb R\cup\{+\infty\}
\end{align*}
 can be viewed as a total energy (the sum of a potential and a kinetic term). The inequality then states that moving from the initial rest position $u$ to the new one $\hat u$ does not increase the total energy.
\end{remark}

\begin{remark}
It follows from Theorem \ref{thm2} that if for every $u\in\dom \Phi$ there exists a sequence $(t_n)_{n\in\mathbb N}\subseteq(0,+\infty)$ converging to zero such that $\operatorname{prox}_{t_n\Phi}(u)\neq\emptyset$, 
then $\dom|\partial\Phi|$ is dense in $\dom \Phi$. 
\end{remark}

\subsection{A couple of consequences}\label{consequences}
We now give a couple of applications of Theorem \ref{thm1}.  As in the previous subsection, we consider a geodesic space $\big(X, d\big)$ and  a proper geodesically convex functional \(\Phi \colon X \to \mathbb{R} \cup \{+\infty\}\) bounded from below.

\subsubsection{The fixed-point theorem of Caristi}
A classical consequence of Ekeland’s variational principle is the so-called Caristi fixed-point theorem \cite[Theorem (2.1)']{C_1976}. A constructive variant of this result can be recovered for geodesically convex functionals by means of Theorem \ref{thm1}.
\begin{proposition}[Caristi fixed-point theorem]\label{P1}
    Let $T: X\to X$ be a mapping. Assume that 
    \begin{align}\label{cariscon}
            \Phi(v) \ge \Phi\big(T(v)\big) + d\big(v, T(v)\big) \quad \forall v\in X. 
    \end{align}
Suppose further that there exist $u,\hat u\in X$ and $t>0$ such that
\begin{align*}
   \Phi(u)<\inf_{v\in X}\Phi(v) + t^{\frac{q}{p}}\quad\text{and}\quad\hat u\in\operatorname{prox}_{t\Phi}(u).
\end{align*}
Then, $\hat u$ is a fixed point of $T$, i.e., $T(\hat u) = \hat u$.
\end{proposition}\noindent
This formulation shows that the existence of minimizers of regularized problems, combined with geodesic convexity and condition \eqref{cariscon}, is sufficient to guarantee fixed-point properties, even without completeness or lower semicontinuity.

\subsubsection{A necessary condition for metric sub-regularity}
Metric sub-regularity is a fundamental stability property of set-valued mappings; it relates the distance to the solution set with a suitable residual measure; see \cite{CDK_2018,DGKO_2020,DR_2014,JOV_2025} for detailed references. This property for convex subdifferentials was first characterized in \cite[Theorem 3.3]{A_2008} through growth conditions via Ekeland’s variational principle. 
Using the metric slope, the global sub-regularity of the subdifferential may be interpreted as a quantitative relation between the slope of a functional and its deviation from the minimum value; see \cite[Remark~6.2]{AC_2014}.
\smallbreak\noindent
The following result provides a necessary condition for the sub-regularity property in terms of an error bound over the domain of the proximal mapping.
\begin{proposition}\label{P2}
Suppose that $\Phi$ has at least one minimizer. 
     If there exists $\kappa>0$ such that 
     \begin{align}\label{subreg}
         \operatorname{dist}\big(u,\argmin_{v\in X}\Phi(v)\big) \le \kappa |\partial\Phi|(u)^{q/p}\quad \forall u\in X,
     \end{align}
     then, for any $t>0$, there holds
     \begin{align}\label{et}
         \Phi(u) \ge\inf_{v\in X}\Phi(v) + \frac{t^{q/p}}{(t^{q/p}+\kappa)^{p}}\hspace{0.04cm} \operatorname{dist}\big(u,\argmin_{v\in X}\Phi(v)\big)^p\quad \forall u\in\dom \operatorname{prox}_{t\Phi}.
     \end{align}
\end{proposition}
\noindent 
The growth condition (\ref{et}) is a well-known error bound; see \cite{AC_2014,AC_2017,K_2015}. In Wasserstein spaces, it is common to refer to conditions such as (\ref{et}) as entropy–transport inequalities or Talagrand inequalities; see \cite[Section~3.8]{HM_2019} and \cite[Section~3.1]{BB_2018}.



\subsection{Some pertinent examples}
We present some relevant examples concerning the applicability of Theorem~\ref{thm1} in situations where the assumptions of the classical Br\o ndsted--Rockafellar theorem are not satisfied.


\subsubsection{Lack of lower semicontinuity}
Let \(H\) be a Hilbert space and \(C\subseteq H\)  a nonempty open bounded convex set.  
Fix \(\xi\in H\) and consider the functional $\Phi:H\to\mathbb R\cup\{+\infty\}$ given by
\begin{align*}
    \Phi(v)=-\langle\xi,v\rangle + \iota_{C}(v),
\end{align*}
where $\iota_C:H\to\mathbb R\cup\{+\infty\}$ denotes the indicator function of $C$. It is clear from its definition that $\Phi$ is proper, convex, and bounded from below. Furthermore, it can be seen that 
\begin{itemize}
        \item[(i)] $\Phi$ is not lower semicontinuous;
         
         \item[(ii)] For any $u\in H$ and $t>0$, 
        \[
    \operatorname{prox}_{t\Phi}(u)=
            \begin{cases}
            \{\,u + t^{\,q-1}\|\xi\|^{\,q-2}\xi\,\} 
                & \text{if }\,\, u \in C - t^{\,q-1}\|\xi\|^{\,q-2}\xi, \\[6pt]
            \varnothing 
                & \text{otherwise.}
    \end{cases}
\]
    \end{itemize}
While the functional is not lower semicontinuous, and the classical Br\o ndsted-Rockafellar theorem cannot be applied, one can still apply Theorem \ref{thm1} for specific choices of parameters. 


\subsubsection{Lack of completeness}
Let $\Omega\subseteq\mathbb R^d$ be a convex bounded open set. Consider the set 
\[
    \mathcal P_2(\Omega):=\left\{\mu\in\mathcal P_2(\mathbb R^d):\, \mu(\Omega)=1\right\}
\]
Preliminaries on the Wasserstein space are given in Subsection \ref{premwass}. The tuple $(\mathcal P_2(\Omega),W_2)$, where $W_2$ is the $2$-Wasserstein distance, forms a  non-complete geodesic metric space. Consider the entropy functional $\Phi:\mathcal P_2(\Omega)\to\mathbb R\cup\{+\infty\}$ given by 
\begin{align*}
    \Phi(\nu):=
    \begin{cases}
        \displaystyle \int_\Omega \rho(x)\log\rho(x)\,dx 
        & \text{if }\,\, \nu = \rho\,dx \\[1.2ex]
        +\infty 
        & \text{otherwise.}
    \end{cases}
\end{align*}
It is clear from definition that $\Phi$ is proper, geodesically convex, and bounded from below. It can further be seen from standard arguments, e.g., \cite[Lemma 3.3.2]{F_2021}, that 
 \begin{itemize}
         \item[(i)] $\Phi$ is lower semicontinuous;
         
         \item[(ii)] For any $\mu\in \mathcal P_2(\Omega)$ and $t>0$, $ \operatorname{prox}_{t\Phi}(\mu)\neq\emptyset.$ 
    \end{itemize}
From this example we see that Theorem \ref{thm1} can replace the usual variational principle on a non-complete geodesic space for certain functionals. 


\subsubsection{Non-existence of proximal points}
Let $(c_{00},\|\cdot\|_{\ell^2})$ be the space of finitely supported real sequences equipped with the $\ell^2$-norm. Consider the functional $\Phi:c_{00}\to\mathbb R\cup\{+\infty\}$ given by 
\[
  \Phi(u)=\sum_{n=1}^\infty \frac{n^2}{2}\big(u_n-n^{-2}\big)^2. 
\]
It is clear that $\Phi$ is proper, convex, and bounded from below. It was proved in \cite[Theorem 2.1-(d)]{W_2025} that it is nowhere sub-differentiable. Using this result, one can verify that 
\begin{itemize}
    \item[(i)] $\Phi$ is lower semicontinuous;

    \item[(ii)] For every $u\in c_{00}$ and $t>0$, $\operatorname{prox}_{t\Phi}(u) = \emptyset$.
\end{itemize}
Item $(i)$ follows from \cite[Theorem 2.1-(d)]{W_2025}, while item $(ii)$ follows from the emptiness of the subdifferential at all points and the sum rule for subdifferentials  \cite[Theorem 9.5.4]{ABM_2014}. 
\smallbreak\noindent
This example shows that certain pathologies of non-complete spaces cannot be avoided; here Theorem \ref{thm1} becomes vacuous as there are no proximal points.


\section{Proofs of subsections \ref{Mr} and \ref{consequences}}\label{proofs}



\subsection{Proof of Theorem \ref{thm1}}
Let $\hat u\in X$ and $\lambda>0$ be such that 
\[
    \hat u\in \argmin_{v \in X} \left\{ \Phi(v) + \frac{\varepsilon}{p\lambda^p} d(v,u)^p \right\}.
\]
\noindent
Let us begin proving item $(a)$; we proceed by contradiction. 
 Suppose that \(d(u, \hat{u}) > \lambda\). Choose a constant-speed geodesic \(\gamma \colon [0,1] \to X\) with \(\gamma(0) = u\) and \(\gamma(1) = \hat{u}\) along which $\Phi$ is convex. Define
\[
t := \frac{\lambda}{d(u, \hat{u})}.
\]
Since we assumed $d(u,\hat u)>\lambda$, it must be that $t\in(0,1)$. The geodesic convexity of \(\Phi\) gives
\begin{align}\label{convf0}
    \Phi(\gamma(t)) \leq (1 - t) \Phi(u) + t \Phi(\hat{u}).
\end{align}
By optimality of \(\hat{u}\), we have
\begin{align*}
    \Phi(\hat{u}) + \frac{\varepsilon}{p\lambda^p} d(\hat{u},u)^p \le \Phi(\gamma(t)) + \frac{\varepsilon}{p\lambda^p} d(\gamma(t),u)^p.
\end{align*}
Combining this with (\ref{convf0}) gives
\begin{align}
    \Phi(\hat{u}) + \frac{\varepsilon}{p\lambda^p} d(\hat{u},u)^p \le (1 - t) \Phi(u) + t \Phi(\hat{u}) + \frac{\varepsilon}{p\lambda^p} d(\gamma(t),u)^p.
\end{align}
As $\gamma$ is a constant-speed geodesic, we have that \(d(\gamma(t),u)=td(\hat u,u)=\lambda\)  , and hence
\begin{align*}
    \Phi(\hat{u}) + \frac{\varepsilon}{p\lambda^p} d( \hat{u},u)^p &\le (1 - t) \Phi(u) + t \Phi(\hat{u}) + \frac{\varepsilon}{p\lambda^p} t^p d(u, \hat{u})^p
\end{align*}
Using that \(\Phi(u) \leq \Phi(\hat{u}) + \varepsilon\) and simplifying, we get 
\begin{align*}
    \Phi(\hat{u}) + \frac{\varepsilon}{p\lambda^p} d( \hat{u},u)^p &\le (1 - t) (\Phi(\hat u)+\varepsilon) + t \Phi(\hat{u}) + \frac{\varepsilon}{p\lambda^p} t^p d(u, \hat{u})^p\\
    &= \Phi(\hat u)+\varepsilon\Big(1-\frac{\lambda}{d(u,\hat u)}\Big) +  \frac{\varepsilon}{p}. 
\end{align*}
Rearranging and simplifying, we obtain
\begin{align*}
    \frac{1}{p\lambda^p} d( \hat{u},u)^p &\le \frac{d(u,\hat u)-\lambda}{d(u,\hat u)} +  \frac{1}{p} .
\end{align*}
Multiplying both sides by $\lambda^{p}$ and rearranging, 
\begin{align}\label{preconome}
     \frac{\lambda-d(u,\hat u)}{d(u,\hat u)}\lambda^{p}&\le    \frac{\lambda^{p}}{p} - \frac{1}{p} d( \hat{u},u)^p
\end{align}
Define $\omega:[0,+\infty)\to[0,+\infty)$ as $\omega(s) = s^p/p$. By convexity of $\omega$, it follows from (\ref{preconome}) that 
\[
     \frac{\lambda-d(u,\hat u)}{d(u,\hat u)}\lambda^{p}  \le \omega(\lambda)-\omega(d(u,\hat u))\le \omega'(\lambda)(\lambda-d(u,\hat u))=\lambda^{p-1}(\lambda-d(u,\hat u)).
\]\noindent
Dividing both sides by $\lambda^{p-1}(\lambda-d(u,\hat u))<0$, we obtain $\lambda/d(u,\hat u)\geq 1$. 
A contradiction. 
\smallbreak\noindent
Let us now proceed with item \((b)\).  By optimality of \(\hat u\),  
\begin{equation}\label{optine2}
    \Phi(\hat u)+\frac{\varepsilon}{p\lambda^p}d(\hat u,u)^p\le \Phi(v)+\frac{\varepsilon}{p\lambda^p}d(v,u)^p\quad\forall v\in X.
\end{equation}
In particular, taking \(v=u\) gives
\begin{equation*}
  \Phi(\hat u) \le  \Phi(\hat u)+\frac{\varepsilon}{p\lambda^p}d(\hat u,u)^p\le \Phi(u).
\end{equation*}
Finally, we prove item $(c)$. 
Fix \(v\in X\setminus\{\hat u\}\) and let \(\gamma:[0,1]\to X\) be a constant--speed geodesic with  \(\gamma(0)=\hat u\) and \(\gamma(1) = v\) along which $\Phi$ is convex. For each \(t\in[0,1]\), set \(v_t:=\gamma(t)\). Then, 
\begin{align}\label{cti}
    \Phi(v_t)\le (1-t)\Phi(\hat u)+t\Phi(v)\quad\text{and}\quad d(v_t,u)\le d(\hat u,u)+t\,d(v,\hat u)\quad \forall t\in(0,1]. 
\end{align}
Plugging \(v_t\) into (\ref{optine2}), using (\ref{cti}), and simplifying yields
\[
t\big(\Phi(\hat u)-\Phi(v)\big)\le \frac{\varepsilon}{\lambda^{p}} \Big(\frac{d(v_{t},u)^{p}}{p}-\frac{d(\hat{u},u)^{p}}{p}\Big). 
\]
As in the proof of item $(a)$, let $\omega:[0,+\infty)\to[0,+\infty)$ be given by $\omega(s) = s^p/p$. By convexity of $\omega$, we see that, for all $t\in(0,1]$,  
\begin{align*}
    t\big(\Phi(\hat u)-\Phi(v)\big)&\le \frac{\varepsilon}{\lambda^{p}} \big(\omega(d(v_t,u)) - \omega(d(\hat u, u))\big)\\
    &\le \frac{\varepsilon}{\lambda^{p}}\omega'(d(v_t,u))(d(v_t,u)-d(\hat u,u))\\
    &= \frac{\varepsilon}{\lambda^{p}} d(v_t,u)^{p-1}(d(v_t,u)-d(\hat u,u)).  
\end{align*}
Using  (\ref{cti}), we get that for all $t\in(0,1]$, 
\begin{align*}
    t\big(\Phi(\hat u)-\Phi(v)\big)&\le \frac{\varepsilon}{\lambda^{p}} d(v_t,u)^{p-1} td(v, \hat u). 
\end{align*}
This leads to $\Phi(\hat u)-\Phi(v)\le\varepsilon/\lambda^{p}d(v_t,u)^{p-1}d(v, \hat u)$ for all $t\in(0,1]$. Letting $t\longrightarrow0^+$, 
\begin{align}\label{ineqts}
     \Phi(\hat u) - \Phi(v)\le\frac\varepsilon{\lambda^{p}}d(\hat u,u)^{p-1}d(v, \hat u).
\end{align}
Combining this with the estimate in item $(a)$ yields
\begin{align*}
	\Phi(\hat u) \le \Phi(v) + \frac{\varepsilon}{\lambda^{p}}d(\hat u,u)^{p-1}d(v,\hat u) \le \Phi(v) + \frac{\varepsilon}{\lambda} d(\hat u, v).
\end{align*}
This finishes the proof, as $v\in X\setminus\{\hat u\}$ was arbitrary. \hfill\(\square\)


\subsection{Proof of Theorem \ref{thm2}}
From Theorem \ref{thm1} we obtain items $(a)$ and $(c)$ directly.  
 Proceeding as in the proof of Theorem \ref{thm1}, one arrives at inequality (\ref{ineqts}) for any $v\in X\setminus\{\hat u\}$. 
From this, the definition of metric slope and (\ref{ineqts}), we obtain 
\begin{align}\label{iNg}
|\partial \Phi|(\hat{u}) &\leq \frac{\varepsilon}{\lambda^p} d(\hat{u},u)^{p-1},
\end{align}
then $\frac{\lambda^p}{\varepsilon} |\partial \Phi|(\hat{u})\leq d(\hat{u},u)^{p-1}.$
Raising both sides to the power $q=\frac{p}{p-1}$ yields
\begin{align}\label{ineq 1}
 \left( \frac{\lambda^p}{\varepsilon} |\partial \Phi|(\hat{u}) \right)^{q}   \leq d(\hat{u},u)^p. 
\end{align}
Using inequality (\ref{ineq 1}) into \begin{equation*}
 \Phi(\hat u)+\frac{\varepsilon}{p\lambda^p}d(\hat u,u)^p\le \Phi(u),
\end{equation*}
it follows
\[
\Phi(\hat{u}) + \frac{\varepsilon}{p\lambda^p} \left( \frac{\lambda^p}{\varepsilon} |\partial \Phi|(\hat{u}) \right)^{q}  \leq \Phi(\hat{u}) + \frac{\varepsilon}{p\lambda^p} d(\hat{u},u)^p \leq \Phi(u).
\]
Thus, we obtain
\[
\Phi(\hat{u}) + \frac{ \lambda^{q} }{p\varepsilon^{\frac{q}{p}}} |\partial \Phi|(\hat{u})^{q} \leq \Phi(u). 
\]
This completes the proof. \hfill\(\square\) 
\subsection{Proof of Proposition \ref{P1}}
Let \(u \in \dom \Phi\) and $t>0$ such that  $\Phi(u) < \inf_{v\in X}\Phi(v) + t^{\frac{q}{p}}$. Let 
\begin{align*}
   \hat u\in\argmin_{v \in X} \left\{ \Phi(v) + \frac{1}{p\hspace{0.02cm}t} d(v,u)^p \right\}.
\end{align*}
If $\Phi(u) = \inf_{v\in X}\Phi(v)$, then $\hat u\in \argmin_{v\in X}\Phi(v)$, and hence by (\ref{cariscon}),  $d(\hat u,T(\hat u))\le \Phi(\hat u)-\Phi(T(\hat u))\le 0$; so $\hat u$ is a fixed point. Suppose now that $\Phi(u) - \inf_{v\in X}\Phi(v)>0$. Define $\varepsilon:=\Phi(u) - \inf_{v\in X}\Phi(v)$. Set $\lambda = \sqrt[p]{t\varepsilon}$ and observe that $(pt)^{-1} = \varepsilon (p\lambda^p)^{-1}$. We see that the conditions in Theorem \ref{thm1} are satisfied. By Theorem \ref{thm1}, we have
\begin{align*}
    \Phi(\hat u)\le \Phi(v) + \frac{\varepsilon}{\lambda} d(v,\hat u) \quad \forall v\in X. 
\end{align*}
Combining this with (\ref{cariscon}),   we conclude that
\begin{align*}
    d(\hat u, T(\hat u)) + \Phi(T(\hat u))\le  \Phi(\hat u)\le \Phi(T(\hat u)) + \frac{\varepsilon}{\lambda} d(T(\hat u),\hat u).
\end{align*}
This can be rewritten as 
\begin{align*}
    \big(1-\frac{\varepsilon}{\lambda}\big)   d(\hat u, T(\hat u)) \le 0. 
\end{align*}
Since $\varepsilon/\lambda = \varepsilon^{1-\frac{1}{p}}/t^{\frac{1}{p}} = (\varepsilon^{\frac{p}{q}}/t)^{\frac{1}{p}}<1$, the result follows.

\subsection{Proof of Proposition \ref{P2}}

Set $\Lambda:=\argmin_{v\in X}\Phi(v)$, assume that (\ref{subreg}) holds and fix $t>0$. We will prove that for any $\delta \in (0,1)$ there holds 
\begin{align}\label{etdelta}
        \Phi(u) -\inf_{v\in X}\Phi(v) \ge \frac{t^{q/p}(1-\delta)^p}{(t^{q/p}+\kappa)^{p}} d\big(u,\Lambda\big)^p \quad \forall u\in\dom \operatorname{prox}_{t\Phi} .
\end{align}
The result will follow by letting $\delta\longrightarrow0^+$. We proceed by contradiction; suppose there exist $\delta>0$ and $u\in\dom\operatorname{prox}_{t\Phi}$ such that (\ref{etdelta}) does not hold. Then, 
    \begin{align*}
        \Phi(u) -\inf_{v\in X}\Phi(v) < \frac{t^{q/p}(1-\delta)^p}{(t^{q/p}+\kappa)^{p}} d\big(u,\Lambda\big)^p .
    \end{align*}
Since $u\in \dom\operatorname{prox}_{t\Phi}$, there exists $\hat u\in X$ such that 
\[
    \hat u\in \argmin_{v\in X}\big\{ \Phi(v) + \frac{1}{pt} d(v, u)^p\big\}. 
\]
 Consider the positive numbers
\[
    \lambda:= \frac{1-\delta}{1+\kappa t^{-q/p}}d(u,\Lambda)\quad \text{and} \quad \varepsilon:=\lambda^{p}/t.
\]
Observe that by construction, 
\begin{align*}
    \frac{\varepsilon}{p\lambda^p} = \frac{1}{pt}\quad \text{and}\quad \frac{\varepsilon}{\lambda} = \frac{\lambda^{p-1}}{t}.
\end{align*}
    Since $\Phi(u)\le\inf_{v\in X}\Phi(v) + \varepsilon$, Theorem \ref{thm2} yields that $d(u,\hat u)\leq\lambda$   and 
    \begin{align*}
 |\partial\Phi|(\hat u) \le \frac{\lambda^{p-1}}{t}. 
 \end{align*}
This and condition (\ref{subreg}) yield
\begin{equation*}
        d\big(\hat u,\Lambda\big) \le \kappa |\partial\Phi|(\hat u)^{q/p} \le \kappa \Big(\frac{\lambda^{p-1}}{t}\Big)^{q/p} = \lambda\kappa t^{-q/p}. 
\end{equation*}
Combining the previous inequality with the estimate $d(u,\hat u)\le\lambda$, we obtain
\begin{align*}
d(u,\Lambda) &\le d(u,\hat u) + d(\hat u,\Lambda)\le \lambda +\lambda\kappa t^{-q/p}\\
&=\lambda(1+\kappa t^{-q/p})=\frac{1-\delta}{1+\kappa t^{-q/p}}d(u,\Lambda)(1+\kappa t^{-q/p})\\
& =(1-\delta)d(u,\Lambda)<d(u,\Lambda). 
\end{align*}
This is the sought contradiction.  \hfill\(\square\)





\section{The theorem in the space of probability measures}\label{Wassres}



\subsection{Preliminaries and notation}\label{premwass}
In this subsection, we recall basic notions from the theory of optimal transport and the structural properties of the Wasserstein space, following \cite[Chapters~5 and~7]{AGS_2008}.

\subsubsection{The space of probability measures}
The space of Borel probability measures $\mu:\mathcal B(\mathbb R^d)\to [0,1]$ is denoted by $\mathcal P(\mathbb R^d)$. For a probability measure  $\mu\in\mathcal P(\mathbb R^d)$ and a Borel mapping  $\mathfrak t:\mathbb R^d\to\mathbb R^d$, the \emph{pushforward} of $\mu$ through $\mathfrak t$ is the probability measure $\mathfrak t_{\#}\mu\in\mathcal P(\mathbb R^d)$ defined by
\[
\mathfrak t_\#\mu(B):=\mu(\mathfrak t^{-1}(B))
\quad\forall B\in\mathcal B(\mathbb R^d).
\]
A probability measure  $\gamma\in\mathcal P(\mathbb R^d\times\mathbb R^d)$ is said to be  a  \emph{coupling} between $\mu\in\mathcal P(\mathbb R^d)$ and $\nu\in\mathcal P(\mathbb R^d)$ if 
\[
{\operatorname{proj}_1}_{\#}\gamma=\mu\quad \text{and}\quad{\operatorname{proj}_2}_{\#}\gamma=\nu,
\]
where $\operatorname{proj}_1,\operatorname{proj}_2:\mathbb R^d\times\mathbb R^d\to\mathbb R^d$ denote the canonical projections.  
The set of all couplings between $\mu$ and $\nu$ is denoted by $\Gamma(\mu,\nu)$. A sequence $(\mu_n)\subset\mathcal P(\mathbb R^d)$ is said to converge  \emph{narrowly} to $\mu\in\mathcal P(\mathbb R^d)$  if  
\[
\int_{\mathbb R^d} \varphi\,\mathrm d\mu_n \longrightarrow \int_{\mathbb R^d} \varphi\,\mathrm d\mu
\quad \forall \varphi\in C_b(\mathbb R^d),
\]
where \(C_b(\mathbb R^d)\) denotes the space of bounded continuous functions $\varphi:\mathbb R^d\to\mathbb R$.



\subsubsection{The Wasserstein space distance}\label{premWass}
Given $\mu\in\mathcal P(\mathbb R^d)$,   $L^2(\mu)$ denotes the  usual space  of square-integrable  Borel functions $f:\mathbb R^d\to\mathbb R$. We denote
\[
    \mathcal P_2(\mathbb R^d):= \{\mu\in\mathcal P(\mathbb R^d):\, \operatorname{id}\in L^2(\mu)^d \}.
\]
This is precisely the space of Borel probability measures with finite second moment. 
The \textit{Wasserstein distance} between $\mu\in\mathcal P_2(\mathbb R^d)$ and $\nu\in\mathcal P_2(\mathbb R^d)$ is defined by
\[
W_2(\mu,\nu)
:=\inf_{\gamma\in\Gamma(\mu,\nu)}
\left(\int_{\mathbb R^d\times\mathbb R^d}|x-y|^2\,\mathrm d\gamma(x,y)\right)^{\!1/2}.
\]
The tuple $(\mathcal P_2(\mathbb R^d),W_2)$ forms a complete and separable metric space.



\subsubsection{Geodesics in the Wasserstein space}
Given $\mu,\nu\in\mathcal P_2(\mathbb R^d)$, we  consider the set 
\[
    \Gamma_o(\mu,\nu):=\left\{ \gamma\in\Gamma(\mu,\nu):\,\, W_{2}(\mu,\nu) ^2=\int_{\mathbb R^d\times\mathbb R^d}|x-y|^2\,\mathrm d\gamma(x,y) \right\}.
\]
Given $\mu,\nu\in\mathcal P_2(\mathbb R^d)$, the set $\Gamma_o(\mu,\nu)$ is always nonempty. 
We say that a Borel mapping $\mathfrak t:\mathbb R^d\to\mathbb R^d$ is an \textit{optimal transport}  between $\mu\in\mathcal P_2(\mathbb R^d)$ and $\nu\in\mathcal P_2(\mathbb R^d)$ if $(\operatorname{id},\mathfrak t)_{\#}\mu$ belongs to $\Gamma_o(\mu,\nu)$. 
\smallbreak\noindent
A curve $\gamma:[0,1]\to\mathcal P_2(\mathbb R^d)$ is said to be a \emph{Wasserstein geodesic} if there exists $\pi\in\Gamma_o(\mu,\nu)$ such that 
\[
    \gamma(t) = \big((1-t)\operatorname{proj}_1 + t\operatorname{proj}_2\big)_{\#}\pi \quad \forall t\in[0,1]. 
\]
It can be seen that Wasserstein geodesics are precisely the constant-speed geodesics in 
$(\mathcal P_2(\mathbb R^d),W_2)$, and that this space is geodesic.



\subsubsection{The regular  subdifferential}
Let $\Phi:\mathcal P_2(\mathbb R^d)\to \mathbb R\cup\{+\infty\}$ be a proper geodesically convex functional such that 
\begin{align}\label{absconass}
    \dom \Phi\subseteq \bigl\{ 
	\mu \in \mathcal{P}_2(\mathbb{R}^d) : 
	\mu = \rho\, dx \text{ for some } \rho \in L^1(\mathbb{R}^d) 
	\bigr\}.
\end{align}
Given $\mu,\nu\in\dom\Phi$ we will denote by  $\mathfrak  t_{\mu}^\nu:\mathbb R^d\to\mathbb R^d$ the unique optimal transport between $\mu$ and $\nu$; the existence and uniqueness is guaranteed by the Brenier-McCann theorem.
The subdifferential $\partial\Phi(\mu)$ of $\Phi$ at a probability measure $\mu\in\mathcal P_2(\mathbb R^d)$ consists of all $\xi\in L^2(\mu)^d$ such that 
\[
    \Phi(\nu) \ge \Phi(\mu) + \int_{\mathbb R^d} \langle \xi(x), \mathfrak t_{\mu}^\nu(x)-x\rangle\,d\mu(x)\quad \forall \nu\in\mathcal P_2(\mathbb R^d).
\]

\subsection{The Br\o ndsted-Rockafellar theorem in the Wasserstein space} 
 Let $\Phi:\mathcal P_2(\mathbb R^d)\to \mathbb R\cup\{+\infty\}$ be a proper geodesically convex functional bounded from below such that (\ref{absconass}) holds. 
 Given $\tau>0$, we consider the set-valued  mapping $J_{\tau}^\Phi:\mathcal P_2(\mathbb R^d)\rightrightarrows\mathcal P_2(\mathbb R^d)$ given by
\[
    J_{\tau}^\Phi(\mu) :=\argmin_{\nu\in\mathcal P_{2}(\mathbb R^d)}\left\{ \Phi(\nu) + \frac{1}{2\hspace{0.02cm}\tau} W_2(\nu,\mu)^2 \right\}.
\]
This is precisely the proximal mapping (for $p=2$) introduced in Subsection~\ref{Mr}; 
it also corresponds to a single implicit--Euler step of the minimizing movement scheme 
for gradient flows, often referred to as the Jordan--Kinderlehrer--Otto (JKO) scheme 
(see~\cite[Proposition 4.1]{JKO_1998} and~\cite[Definition~2.0.6]{AGS_2008}).
\smallbreak\noindent
We now formulate the metric analogue of the Br\o ndsted--Rockafellar theorem in the Wasserstein framework, and give a couple of remarks before the proof. 
\begin{theorem}\label{thmW}
    Let $\mu\in\mathcal P_2(\mathbb R^d)$ and $\varepsilon>0$ such that 
    \[
        \Phi(\mu) \le \inf_{\nu\in\mathcal P_2(\mathbb R^d)}\Phi(\nu) + \varepsilon. 
    \]
    Let $\lambda>0$ and suppose there exist   $\hat\mu\in \mathcal P_2(\mathbb R^d)$ and $\xi\in L^2(\hat\mu)^d$ such that
    \[
        \hat\mu\in J^\Phi_{\frac{\lambda^2}{\varepsilon}}(\mu)\quad \text{and}\quad \mathfrak t_{\hat \mu}^\mu = \id + \frac{\lambda^2}{\varepsilon}\xi. 
    \]
    Then the following statements are satisfied.  
    	\begin{align*}
		(a)\,\,\, W_2(\mu,\hat \mu)\le\lambda\qquad\,\,\,\, (b)\,\,\, \Phi(\hat\mu) + \frac{\lambda^2}{2\varepsilon}\|\xi\|^2_{L^2(\hat\mu)^d}\le \Phi(\mu)  \qquad\,\,\,\, (c)\,\,\,\xi\in\partial\Phi(\hat \mu)\qquad \,\,\,\,(d)\,\,\,\|\xi\|_{L^2(\hat\mu)^d} \leq \frac{\varepsilon}{\lambda}.
	\end{align*}
\end{theorem}\noindent
\begin{proof}
    Item $(a)$ follows directly from Theorem \ref{thm1}-(a). Item $(c)$ follows from the first-order optimality condition for the proximal problem 
\[
\frac{\varepsilon}{\lambda^2}\big(\mathfrak t_{\hat\mu}^{\mu}-\id\big)\in\partial\Phi(\hat \mu).
\]
Using that $\mathfrak t_{\hat\mu}^{\mu}=\id+\frac{\lambda^2}{\varepsilon}\xi$, we obtain $\xi\in\partial\Phi(\hat\mu)$.
  Observe that 
\begin{align}\label{iNW}
    \|\xi\|_{L^2(\hat\mu)^d} = \frac{\varepsilon}{\lambda^2}\|\mathfrak t_{\hat \mu}^\mu - \id\|_{L^2(\hat\mu)^d}= \frac{\varepsilon}{\lambda^2}W_2(\hat\mu,\mu).
\end{align}
    Since $W_2(\hat \mu,\mu)\le\lambda$, equality (\ref{iNW}) yields immediately item $(d)$. Finally, item $(b)$ can be obtained by repeating the argument in the proof of Theorem \ref{thm2} replacing (\ref{iNg}) with (\ref{iNW}).
\end{proof}

\begin{remark}
	Interpreting $\Phi$ as a potential energy functional on the Wasserstein space and 
	$\xi\in L^2(\nu)^d$ as the velocity field acting on the particles of a probability measure $\nu\in\mathcal P_2(\mathbb R^d)$,  
	inequality in Theorem~\ref{thmW}-(b) expresses an  energy balance. The quantity 
	\[
	    \mathcal E(\nu) := \Phi(\nu) + \frac{\lambda^2}{2\varepsilon}\|\xi\|_{L^2(\nu)^d}^2
	\]
	can be viewed as a total energy, consisting of a potential part and a kinetic term 
	associated with the transport cost. Inequality in Theorem~\ref{thmW}-(b) therefore states 
	that a single JKO step does not increase the total energy, reflecting the dissipation principle 
	of gradient flows.
\end{remark}

\subsection{Entropy-transportation and metric sub-regularity}
 Let $\Phi:\mathcal P_2(\mathbb R^d)\to \mathbb R\cup\{+\infty\}$ be a proper geodesically convex functional such that (\ref{absconass}) holds. 
 \smallbreak\noindent
 Given $\mu\in\mathcal P_2(\mathbb R^d)$ and $C\subseteq\mathcal P_2(\mathbb R^d)$, we denote 
 \[
    W_2\big(\mu,C\big):=\inf_{\nu\in C}W_2(\mu,\nu)\quad\text{and}\quad d_{L^2}\big(0,\partial\Phi(\mu)\big):=\inf_{\xi\in\partial\Phi(\mu)}\|\xi\|_{L^2(\mu)^d}.
 \]
The first quantity can be interpreted as a gauge of how close a given probability measure is to belonging to a specified set, 
while the second---being the minimal \(L^2\)-norm among all elements of the subdifferential of \(\Phi\) at \(\mu\)---serves as a gauge of how close \(\mu\) is to having zero as a subgradient, and hence to being a minimizer of \(\Phi\).
\smallbreak\noindent
As a consequence of Theorem~\ref{thmW}, we obtain an equivalence between the metric sub-regularity 
of the subdifferential and a quantitative growth condition of entropy–transport type.
\begin{theorem}\label{Ent_Loj}
    Assume that $\argmin_{\nu\in\mathcal P_2(\mathbb R^d)}\Phi(\nu)\neq\emptyset$ and that there exists $\tau>0$ such that $J_\tau^\Phi(\mu)\neq\emptyset$ for every $\mu\in \dom\Phi$. Then, the following statements are equivalent.
    \begin{itemize}
        \item[$(i)$] There exists $\kappa>0$ such that 
        \begin{align}\label{subregW}
            W_2\big(\mu,\argmin_{\nu\in \mathcal P_2(\mathbb R^d)}\Phi(\nu)\big)\le \kappa \hspace{0.02cm} d_{L^2}\big(0,\partial\Phi(\mu)\big)\quad \forall \mu\in\mathcal P_2(\mathbb R^d).
        \end{align}
        \item[$(ii)$] There exists $c>0$ such that 
        \begin{align}\label{growthW}
            \Phi(\mu) \ge\inf_{\nu\in \mathcal P_2(\mathbb R^d)} \Phi(\nu) \ + c\hspace{0.04cm}W_2\big(\mu,\argmin_{\nu\in \mathcal P_2(\mathbb R^d)}\Phi(\nu)\big)^2 \quad \forall \mu\in\mathcal P_2(\mathbb R^d).
        \end{align}
    \end{itemize}
\end{theorem}
\noindent
\begin{proof}
    Set $\Lambda := \argmin_{\nu \in \mathcal P_2(\mathbb R^d)} \Phi(\nu)$.  
\smallbreak\noindent
\noindent\emph{(ii) $\implies$ (i).}  
Fix $\mu\in\mathcal P_2(\mathbb R^d)$.  
If $\partial\Phi(\mu)=\varnothing$, then $d_{L^2}(0,\partial\Phi(\mu))=+\infty$ and the inequality is trivial.  
Assume henceforth $\partial\Phi(\mu)\neq\varnothing$.  
Fix $\xi\in\partial\Phi(\mu)$.  
By the definition of the  subdifferential we have
\begin{align*}
    \Phi(\nu)\ge \Phi(\mu) + \int_{\mathbb R^d}\langle \xi(x), \mathfrak t_{\mu}^{\nu}(x)-x\rangle\,d\mu(x)\quad\forall \nu\in\mathcal P_2(\mathbb R^d).
\end{align*}
From this, 
\[
    \Phi(\mu)-\Phi(\nu)
    \le \int_{\mathbb R^d}\langle \xi(x), x-\mathfrak t_{\mu}^{\nu}(x)\rangle\,d\mu(x)
    \le \|\xi\|_{L^2(\mu)^d}\,W_2(\mu,\nu).
\]
Taking the infimum over $\nu\in\Lambda$, we obtain
\[
    \Phi(\mu)-\inf_{\nu\in\mathcal P_2(\mathbb R^d)}\Phi(\nu)
    \le \|\xi\|_{L^2(\mu)^d}\,W_2(\mu,\Lambda).
\]
Combining this with (\ref{growthW})  gives $ c\,W_2(\mu,\Lambda)\le \|\xi\|_{L^2(\mu)^d}$. 
Taking the infimum over $\xi\in\partial\Phi(\mu)$ gives
the desired inequality with $\kappa = 1/c$.
\smallbreak\noindent
\noindent\emph{(i) $\implies$ (ii).}  
Fix $\mu\in\mathcal P_2(\mathbb R^d)$ and $\mu^\tau\in J_\tau^\Phi(\mu)$.  
By Theorem~\ref{thmW}, there exists $\xi\in\partial\Phi(\mu^\tau)$ such that
\begin{align}\label{thmWrelations}
    \mathfrak t_{\mu^\tau}^{\mu} = \id + \tau\,\xi,
    \qquad
    \|\xi\|_{L^2(\mu^\tau)^d} = \frac{1}{\tau}W_2(\mu,\mu^\tau),
    \qquad
    \Phi(\mu) - \Phi(\mu^\tau) \ge \frac{1}{2\tau}W_2(\mu,\mu^\tau)^2.
\end{align}
Using (\ref{subregW}) and  \eqref{thmWrelations}, we get
\[
    W_2(\mu^\tau,\Lambda)\le \kappa\,d_{L^2}\big(0,\partial\Phi(\mu^\tau)\big)
    \le \kappa\,\|\xi\|_{L^2(\mu^\tau)^d}
    = \frac{\kappa}{\tau}\,W_2(\mu,\mu^\tau).
\]
By the triangle inequality,
\[
    W_2(\mu,\Lambda)
    \le W_2(\mu,\mu^\tau)+W_2(\mu^\tau,\Lambda)
    \le \Big(1+\frac{\kappa}{\tau}\Big)W_2(\mu,\mu^\tau).
\]
We deduce from \eqref{thmWrelations} that
\[
    \Phi(\mu) - \inf_{\nu\in\mathcal P_2(\mathbb R^d)}\Phi(\nu)
    \ge \frac{1}{2\tau}W_2(\mu,\mu^\tau)^2
    \ge \frac{\tau}{2(\tau+\kappa)^2}\,W_2(\mu,\Lambda)^2.
\]
Hence (\ref{growthW}) holds with $c = 2^{-1}\tau(\tau+\kappa)^{-2}$.
\end{proof}

\begin{remark}
    In Theorem~\ref{Ent_Loj} we assume uniform existence of minimizers for the JKO steps.  
This assumption is sometimes adopted in the literature to simplify the analysis and to avoid 
technical complications concerning the well-posedness of the minimizing movement scheme; 
see, for instance, \cite[Assumption~(10.1.1b)]{AGS_2008}.
\end{remark}

\begin{remark}
There are several examples of functionals on Wasserstein spaces satisfying condition~\eqref{growthW} (and consequently (\ref{subregW})), for instance, those given by the sum of internal and potential energies for specific parameter choices; we refer to \cite[Section~3]{BB_2018} and \cite[Section~4.2]{HM_2019} for a comprehensive account.
\end{remark}



\appendix
\section{The particular case of normed spaces}\label{A}
In this subsection of the appendix, we specialize the metric constructions of Section~\ref{Mr} to the classical setting of a real normed space. In this framework, the optimality condition for a proximal step admits a concrete expression in dual variables, and it becomes transparent how the choice of a proximal point produces an approximate minimizer together with a subgradient of controlled size.
\smallbreak\noindent
The following theorem can be regarded as an extension of \cite[Theorem 4.1]{C_2023} from Hilbert spaces to general normed spaces. Its proof is a direct consequence of Theorem \ref{thm1}.
\begin{theorem}\label{thmn}
Let $(X,\|\cdot\|)$ be a real normed space, and $\Phi:X\to\mathbb R\cup\{+\infty\}$ a proper convex functional. Fix $p,q\in(1,+\infty)$ such that $p^{-1}+q^{-1}=1$. 
Let \(u \in X\) and $\varepsilon>0$ be such that 
\[
\Phi(u) \leq \inf_{v \in X} \Phi(v) + \varepsilon.
\]
Let $\lambda>0$ be given.  Suppose that $\hat u\in X$ and $\xi\in \partial\Phi(\hat u)$ satisfy
\begin{align}\label{HBd}
    \langle\xi,u-\hat u\rangle = \|\xi\|\|u-\hat u\|
      \,\,\,\,\text{and}\,\,\,\, \|\xi\|=\frac{\varepsilon}{\lambda^p}\|u-\hat u\|^{p-1}
\end{align}
Then, the following estimates are satisfied. 
\begin{align*}
    (a)\,\,\, \|u-\hat u\|\le\lambda\qquad\,\,\,\, (b)\,\,\,\Phi(\hat{u}) + \frac{\lambda^q}{p\varepsilon^{\frac{q}{p}}}\|\xi\|^q \leq \Phi(u)\qquad \,\,\,\,(c)\,\,\, \|\xi\|\le\frac\varepsilon\lambda. 
\end{align*}
\end{theorem}
\begin{proof}
Let us begin with item $(a)$. Condition (\ref{HBd}) is equivalent to 
\[
    0\in \partial \Phi(\hat u) + \frac{\varepsilon}{\lambda^p}J_p(\hat u-u),
\]
where $J_p:X\rightrightarrows X^*$ denotes the duality mapping. This in turn implies  $\hat u\in \operatorname{prox}_{\frac{\lambda^p}{\varepsilon}\Phi}(u)$. Item $(a)$ then follows from Theorem~\ref{thm1}. Using the identities $\|\xi\|=\frac{\varepsilon}{\lambda^p}\|u-\hat u\|^{p-1}$ and
$(p-1)q=p$, we obtain 
\[
\frac{\varepsilon}{p\lambda^p}\|u-\hat u\|^{p}
=\frac{\lambda^q}{p\varepsilon^{q/p}}\|\xi\|^{q}.
\]
Combining this with the fact that $\hat u\in \operatorname{prox}_{\frac{\lambda^p}{\varepsilon}\Phi}(u)$ yields then item $(b)$. Finally,
\[
\|\xi\|=\frac{\varepsilon}{\lambda^p}\|u-\hat u\|^{p-1}
\le \frac{\varepsilon}{\lambda^p}\lambda^{p-1}=\frac{\varepsilon}{\lambda},
\]
whence item $(c)$ follows. 
\end{proof}

\begin{remark}
In the Hilbertian case with the choice $p=2$, the duality mapping coincides with the identity under the canonical identification of $X$ with its dual. In this situation, condition \eqref{HBd} reduces to the construction in \cite[Theorem 4.1]{C_2023}.
\end{remark}

\begin{remark}
Item $(b)$ strengthens the classical estimate $\Phi(\hat u)\le \Phi(u)$ by giving an explicit lower bound on the decrease $\Phi(u)-\Phi(\hat u)$ in terms of the size of a subgradient at $\hat u$. In particular, if $\hat u$ is not a critical point, then the proximal construction yields a strict decrease, that is, $\Phi(\hat u)<\Phi(u)$.
\end{remark}

\section{The limiting case $p=+\infty$}\label{B}
This part of the appendix records the limiting regime in which the proximal regularization becomes a hard constraint. As the exponent in the powered-distance kernel tends to infinity, the penalized problems enforce confinement to a fixed ball; in the sense of $\Gamma$--convergence the kernels converge to the indicator function of a ball. Accordingly, the approximate minimizer produced by the proximal step is replaced by a minimizer under a ball constraint, and the same Ekeland-type estimate is recovered.

\begin{theorem}\label{thmpinfty}
Let $\big(X, d\big)$ be a geodesic space, and \(\Phi \colon X \to \mathbb{R} \cup \{+\infty\}\) a proper geodesically convex functional.  Let \(u \in X\) and $\varepsilon>0$ be such that 
\[
\Phi(u) \leq \inf_{v \in X} \Phi(v) + \varepsilon.
\]
Assume that there exist $\hat u\in X$ and $\lambda>0$ such that 
\[
    \hat u\in \argmin_{v \in X} \left\{\Phi(v) + \iota_{\mathbb B(u,\lambda)}(v)\right\}.
\]
Then the following estimates are satisfied. 
\begin{align*}
    (a)\,\,\, d(u,\hat u)\le\lambda\qquad\,\,\,\, (b)\,\,\,\Phi(\hat{u})  \leq \Phi(u)\qquad \,\,\,\,(c)\,\,\, \Phi(\hat u) \le \Phi(v) + \frac{\varepsilon}{\lambda} d(v,\hat u)\quad \forall v\in X. 
\end{align*}
\end{theorem}
\begin{proof}
    Items $(a)$ and $(b)$ follow trivially. Let $v\in X$; the computation is divided in two cases. If \(d(u,v)\le \lambda\), then \(v\in\mathbb B(u,\lambda)\) and by minimality of \(\hat u\) on
\(\mathbb B(u,\lambda)\) we have \(\Phi(\hat u)\le \Phi(v)\), hence
\[
\Phi(\hat u)\le \Phi(v)\le \Phi(v)+\frac{\varepsilon}{\lambda}d(v,\hat u).
\]
If \(d(u,v)>\lambda\), let \(\gamma:[0,1]\to X\) be a constant-speed geodesic from \(u\) to \(v\) along which $\Phi$ is convex, and set
\(t:=\lambda/d(u,v)\). Then \(\gamma(t)\in\mathbb B(u,\lambda)\), so
\(\Phi(\hat u)\le \Phi(\gamma(t))\). By geodesic convexity,
\[
\Phi(\hat u)\le \Phi(\gamma(t))\le (1-t)\Phi(u)+t\Phi(v)\le (1-t)(\Phi(v)+\varepsilon)+t\Phi(v)
=\Phi(v)+(1-t)\varepsilon.
\]
Since \(1-t=(d(u,v)-\lambda)/d(u,v)\le (d(u,v)-\lambda)/\lambda\), it follows that
\[
\Phi(\hat u)\le \Phi(v)+\frac{\varepsilon}{\lambda}\,(d(u,v)-\lambda).
\]
Finally, \(d(u,v)\le d(u,\hat u)+d(\hat u,v)\le \lambda+d(\hat u,v)\) by (a), hence
\(d(u,v)-\lambda\le d(v,\hat u)\), and therefore
\[
\Phi(\hat u)\le \Phi(v)+\frac{\varepsilon}{\lambda}d(v,\hat u).
\]
This proves item $(c)$ and completes the proof. 
\end{proof}

\section{The case of general proximal kernels}\label{C}

\noindent
In this subsection of the appendix, we present a more general version of Theorem \ref{thm1}, in which powered-distance perturbations are replaced by more general ones.
\smallbreak\noindent
Let $\omega:[0,+\infty)\to[0,+\infty)$ be a convex  function satisfying 
\begin{align*}
    \omega(0)=0\qquad\text{and}\qquad \omega(s)>0\quad \forall s>0.
\end{align*}
These properties imply the existence of the right-derivative $\omega'_+:[0,+\infty)\to[0,+\infty)$ and furthermore that $\omega_+'(s)>0$ for all $s>0$. 
We are now ready to state the general theorem.
\begin{theorem}\label{thmgenper}
Let $\big(X, d\big)$ be a geodesic space, and \(\Phi \colon X \to \mathbb{R} \cup \{+\infty\}\) a proper geodesically convex functional.  Let \(u \in X\) and $\varepsilon>0$ be such that 
\[
\Phi(u) \leq \inf_{v \in X} \Phi(v) + \varepsilon.
\]
Assume that there exist $\hat u\in X$ and $\lambda>0$ such that 
\[
    \hat u\in \argmin_{v \in X} \left\{ \Phi(v) + \frac{\varepsilon}{\lambda \hspace{0.02cm} \omega'_+(\lambda)}\omega\big(d(v,u)\big) \right\}.
\]
Then the following estimates are satisfied. 
\begin{align*}
    (a)\,\,\, d(u,\hat u)\le\lambda\qquad\,\,\,\, (b)\,\,\,\Phi(\hat{u})  \leq \Phi(u)\qquad \,\,\,\,(c)\,\,\, \Phi(\hat u) \le \Phi(v) + \frac{\varepsilon}{\lambda} d(v,\hat u)\quad \forall v\in X. 
\end{align*}
\end{theorem}
\begin{proof}
    The proof is analogous to that of Theorem \ref{thm1}.
\end{proof}

\begin{remark}
  Examples of functions $\omega:[0,+\infty)\to[0,+\infty)$ satisfying the hypotheses above include 
power-type (e.g.,\ $\omega(s) = s^p/p$ for $p\ge1$), 
exponential-type (e.g.,\ $\omega(s) = e^s - 1$), 
logarithmic-type (e.g.,\ $\omega(s) = s\log(1+s)$), 
and hyperbolic-type (e.g.,\ $\omega(s) = \log\cosh s$). 
We refer to \cite{BK_2018} for the study of proximal operators with Young moduli of continuity in uniformly convex Banach spaces.
\end{remark} 
\bibliography{references}{}
\bibliographystyle{plain}
\nocite{*}

\end{document}